\documentclass{amsart}

\usepackage{amssymb,amsmath,amsthm} 
\usepackage{graphicx}
\usepackage{xypic}
\usepackage{listings}
\usepackage[parfill]{parskip}
\usepackage{cite}

\numberwithin{equation}{section}

\newtheorem{thm}{Theorem}[section]
\newtheorem{lem}[thm]{Lemma}
\newtheorem{prop}[thm]{Proposition}

\theoremstyle{definition}

\theoremstyle{remark}
\newtheorem{remark}[thm]{Remark}

\DeclareMathOperator{\sgn}{sgn}

\DeclareMathOperator{\CT}{CT}
\DeclareMathOperator{\AV}{AV}

\newcommand{\bC}{\mathbf{C}}
\newcommand{\bj}{\mathbf{j}}

\newcommand{\bN}{\mathbf{N}}
\newcommand{\ba}{\overline{\alpha}}

\newcommand{\bt}{\tilde{b}}

\newcommand{\gerg}{\mathfrak{g}}
\newcommand{\gerh}{\mathfrak{h}}
\newcommand{\reg}{\mathrm{reg}}
\newcommand{\Rad}{\mathrm{Rad}}
\newcommand{\ld}{\lambda}
\newcommand{\br}{\mathbf{r}}
\newcommand{\binf}{\mathbf{\infty}}
\newcommand{\bM}{\mathbf{M}}

\newcommand{\shat}[1]{{#1}^{'}}

\begin{document}

\title{The Bernstein-Sato $b$-function of the Space of Cyclic Pairs}

%    Remove any unused author tags.

%    author one information
\author{Robin Walters}
\address{Department of Mathematics, University of Chicago}
\curraddr{5734 S. University Avenue, Chicago, IL 60637}
\email{robin@math.uchicago.edu}
%\subjclass[2000]{Primary }
%    For articles to be published after 1 January 2010, you may use
%    the following version:
\subjclass[2010]{Primary 14F10, Secondary 32S25}

\date{April 16, 2014}

\begin{abstract}
We compute the Bernstein-Sato polynomial of $f$, a function which given a pair $(M,v)$ in $X = M_n(\bC) \times \bC^n$ tests whether $v$ is a cyclic vector for $M$.  The proof includes a description of shift operators corresponding to the Calogero-Moser operator $L_k$ in the rational case. 
\end{abstract}

\maketitle

\section{Introduction}

Let $f$ be an algebraic function on a variety $X$ over $\bC$.  Let $D_X$ be the ring of algebraic differential operators on $X$.  The Bernstein-Sato $b$-function of $f$ is defined to be the minimal degree monic function $b(s)$ in $\bC[s]$ such that 
\begin{equation}\label{sato-bernstein}
D f^{s+1} = b(s) f^{s}
\end{equation}
for some operator $D$ in $D_X[s] = \bC[s] \otimes D_X $.  We call $D$ the \textit{Bernstein operator} and (\ref{sato-bernstein}) the \textit{Bernstein equation}.  A minimal $b(s)$ must exist since the set of all $b(s)$ satisfying (\ref{sato-bernstein}) form an ideal in $\bC[s]$.  Existence of non-zero solutions to (\ref{sato-bernstein}) was proved by Bernstein in 1971 \cite{Bern}.  The rationality of the roots of $b(s)$ was proved by Kashiwara in 1976 \cite{kashiwara1976b}.  The $b$-function is interesting, in part, because it is an invariant of the singularities of the divisor given by $f$. 

In \cite{Op5}, Opdam proves a conjecture of Yano and Sekiguchi \cite{YS} by computing the $b$-function corresponding to $I = \prod_{\alpha \in R^+}\alpha^2 $, a $W$-invariant function on $\gerh$.  In type $A_n$, the function $I$ is the square of the Vandermonde determinant.  Opdam proves the result by realizing the Bernstein operator as a shift operator related to the Calogero-Moser operator.

For this paper, let $X = M_n(\bC) \times \bC^n$.  For $(M,v) \in X$, we say $v$ is \textit{cyclic} for $M$ or that $(M,v)$ is a \textit{cyclic pair} if the set $\lbrace v, Mv, M^2v,\ldots \rbrace$ is a spanning set of $\bC^n$.  Let $C(M,v)$ denote the square matrix 
\[
[v\ Mv\ M^2v\ \ldots\ M^{n-1}v].
\]
We define $f(M,v) = \det (C(M,v))$, a polynomial on $X$.  Then $(M,v)$ is cyclic if and only if $f(M,v) \not = 0$.  

The function $f$ maps to $I$ via radial reduction.  Since radial reduction maps the standard Laplacian operator $\Delta$ to the Calogero-Moser operator on $\gerh$ \cite{GGS}, we can think of $X$ and $f$ as describing a broader, yet simpler, precursor situation to the one studied in \cite{Op5}. More generally, the function $f$ and the space $X$ are relevant to the study of mirabolic $D$-modules and rational Cherednik algebras \cite{mir3}\! - \! \nocite{mir1}\!\!\!\!\! \cite{mir2}.

The main result of the paper is the computation of the $b$-function of $f$. 

\begin{thm}\label{thm-b-func}
The $b$-function of $f$ is 
\begin{equation}
\bt(s) = \prod_{0 \leq c < d \leq n} ( s + 1 + c/d) . 
\end{equation}
\end{thm}

 Since $f$ is a semi-invariant, the calculation has some similarity to the prehomogeneous case originally considered by Sato in \cite{sato1} and \cite{sato2}. In that case, $G$ acts on a vector space with open dense orbit.  However, our space is not prehomogeneous, so some additional work is required.  

The proof will proceed in three parts.  First we will define a differential operator $S$ in $D_X[s]$, and show that one has an equation  
\begin{equation} \label{eqn-hatb}
S f^{s+1} = \shat{b}(s) f^s 
\end{equation}
for an unknown function $\shat{b}(s)$ in $\bC[s]$. 

Secondly, we show $\bt(s)$, our specific candidate function, is the monic associate of $\shat{b}(s)$.

Thirdly, we will show that $\bt(s)$ has, in fact, the \textit{minimal} degree, completing the proof.

Section \ref{sec-f} contains some results about $f$ and cyclic vectors.  Then we continue to the proof of our main theorem.  This is the content of Sections \ref{sec-proof1} through \ref{sec-proof4}.  In the Appendix (Section \ref{Section-proof-prop}), we will give a proof characterizing the structure of the space of shift operators of the rational Calogero-Moser operator, an analog of a similar result for the trigonometric case in \cite{Op5}.  

\section{Cyclic Vectors and Semi-Invariants}
\label{sec-f}

The following result is well-known.

\begin{prop}
The matrix $M$ has a cyclic vector if and only if each Jordan block $B_i$ has a distinct eigenvector $\lambda_i$, that is, if each eigenspace is one-dimensional.
\end{prop}

We conjecture that $f$ is irreducible.  The above result allows us to prove something weaker but still sufficient for our purposes.  Let $\gerg = M_n(\bC)$.

\begin{prop} \label{prop-f-nofactor}
The function $f$ has no non-constant, proper factor $h$ where $h \in \bC[\gerg]$.  
\end{prop}
\begin{proof}
Such a proper factor would correspond to a set of codimension 1 in $\gerg$ containing matrices $M$ with no cyclic vector.  By the assumption that at least two Jordan blocks share an eigenvalue, we have $k-1$ choices of $\lambda_i$.  Assuming $\lambda_1 = \lambda_2 = \lambda$ and reordering the basis so that the true eigenvectors of $B_1$ and $B_2$ are the first two vectors followed by all the generalized eigenvectors, we see that $B_1 \oplus B_2$ is 
\[
\left( 
\begin{array}{c|c} 
\lambda \mathrm{Id}_2 & A   \\ \hline
0 & B \end{array}\right).
\]
The submatrix $B$ is upper triangular with $\lambda$ on its diagonal, and the entries of $A$ and of $B$ above the diagonal are 1 or 0.  Then this matrix is stabilized by 
\[
\left( 
\begin{array}{c|c} 
 GL(2) & 0   \\ \hline
0 & \bC^* \mathrm{Id_{n-2}} \end{array}\right).
\]
Thus the dimension of this partition is less than or equal to
\[
n^2 + (k-1) - (k-2) - 1 - 4 = n^4 - 4 .
\]
So we are done.
\end{proof}

 Given a space $Y$ with a $G$-action, we denote the $G$-invariant operators on $Y$ as $D_Y^G$.  Denote $p^g(y) = p(g\cdot y)$.  The function $p$ is said to be a $G$-semi-invariant corresponding to character $\chi$ if 
\[
p^g = \chi(g) p \ \ \ \text{for all} \ \ g \in G . 
\]
We denote the semi-invariant functions corresponding to $\chi$ as $\bC[Y]^\chi$ and differential operators corresponding to $\chi$ as $D_Y^\chi$.  

In our case, we have $G = GL_n(\bC)$ acting on $X = \mathfrak{gl}_n (\bC) \times \bC^n$ via conjugation on the first factor and multiplication on the second factor. 
The one-dimensional representations (or characters) are just powers of the determinant $\chi = \det^r$.  

Note that $f \in \bC[X]^{\det}$.

The space of diagonal matrices in $\mathfrak{gl}_n (\bC)$ is isomorphic to $\bC^n$.  We denote it by $\gerh$. Let $Z = \gerh \times V \subset X.$ 

The following result is originally due to Weyl, although we give a different proof. 

\begin{prop} \label{prop-f-structure}
Let $p \in \bC(s)[X]^\chi$ where $\chi = \det^r$.   Then 
\[
p = f^r h,
\]
where $h \in \bC(s)[\gerg]^G$.
\end{prop}
\begin{proof}
Consider $q = p f^{-r} \in \bC(s)(X)^G$. 

Now define the matrix 
\[
d = \mathrm{diag}(1,\ldots,1,\lambda,1\ldots,1) \in G
\]
containing $\lambda$ in its $i^\mathrm{th}$ entry.  This acts on $Z$ by fixing $\gerh$ and scaling $v^i$, the $i^{\mathrm{th}}$ coordinate of $V$.  However, since $q|_Z$ is invariant, it must be homogeneous of degree $0$ in $v^i$, that is, it is independent of $v^i$.  By invariance, $q$ is independent of $v^i$ on $G \cdot Z$, which is dense in $X$, and thus $q \in \bC(s)(\gerg)^G$.

It is not hard to see that  
\[
\bC(s)(\gerg)^G = \left\{ \frac{a}{b}\ \middle|\ a,b \in \bC(s)[\gerg]^G \right\}
\]
since a general element of $\bC(s)(\gerg)^G$ is a ratio of semi-invariants, but $\bC(s)[\gerg]$ has no non-invariant semi-invariants.

So we know 
\[
p = f^r \frac{a}{b}
\]
where $a,b \in \bC(s)[\gerg]^G$ and are, we can assume, relatively prime.  This equation implies that $b$ divides $f^r$.  However, by Proposition \ref{prop-f-nofactor}, $f^r$ has no factor in $\bC[\gerg]$ and thus none in $\bC(s)[\gerg]$ either.
\end{proof}

%%%%%%%%%%%%%%%%%%%%%%%%%%%%%%%%%%%%%%%%%%%%%%%%%%%%%%%%%

\section{The operator $S$}
\label{sec-proof1}

When we use the term \textit{order} in reference to a differential operator in $D_X[s]$, we refer to the traditional filtration in which $\partial_{x_i} = \frac{\partial}{\partial x_i}$ has degree 1 and functions in $\bC[s][X]$ have degree 0.  We will also refer to a $\mathbf{Z}$-grading on $D_X[s]$ in which $|\partial_{x_i}| = -1$ and $|x_i| = 1$.  The ring $\bC[s]$ lives in the grade 0.   Note that the grading is well-defined on $D_X[s]$ since it respects the defining relation $\partial_{x_i} x_i - x_i \partial_{x_i} = 1$.  When we refer to this grading, we will use term \textit{degree}. 

We define a differential operator $S$ which is given by taking the function $f$ and replacing all variables with the corresponding partial derivatives. That is,
\[
 S = \det([\partial_v\ \partial_M\partial_v\ \partial_M^2\partial_v\ \ldots\ \partial_M^{n-1}\partial_v])
\]
where $[\partial_v]_i = \partial_{v_i}$ and $[\partial_M]^i_{j}=\partial_{m^i_{j}}.$
Then $S$ is an order $\frac{n(n+1)}{2}$ differential operator and $S \in D_X^{\det^{-1}}$.  The definition of $S$ here is analogous to that of the Bernstein operator in Sato's prehomogeneous case.

We can now prove the following:  

\begin{prop}\label{prop-const}
There exists a function $\shat b \in \bC[s]$, such that $S f^{n+1} = \shat b(s) f^n$.  Further $\deg(\shat b(s)) \leq \frac{n(n+1)}{2}$.
\end{prop}

\begin{proof}
Since $S$ has weight $\det^{-1}$ and $f^{s+1}$ has weight $\det^{s+1}$, the operator $S \circ \hat f^{s+1}$ has weight $\det^{s}$.  Applying this operator to $1$, we get a polynomial $p \in \bC[s][X]^{\det^s}$.  Therefore by Proposition \ref{prop-f-structure}, there exists $q \in \bC[s][X]^G$ such that 
\[
S f^{s+1} = f^s q.
\]
Taking the degree of both sides we see
\[
-\frac{n(n+1)}{2} + (s+1) \frac{n(n+1)}{2} = s \frac{n(n+1)}{2} + \deg(q).
\]
Thus $\deg(q) = 0$ and so $q \in \bC[s]$ as desired.  Denote $q$ as $b'(s)$.  Moreover since $S$ is an $\frac{n(n+1)}{2}$-order operator, we see $\deg(b') \leq \frac{n(n+1)}{2}$. 
\end{proof}

\section{Localization}
\label{sec-proof2}

In this and the next section, we will show that $\shat b(s)$ equals
\begin{equation}
\tilde b (s) \alpha_n = \prod_{0 \leq c < d \leq n} (d(s+1) + c),
\end{equation}
where $\alpha_n = \left( \prod_{d=1}^n d^d \right) \in \bC$.
We factor this as:
\begin{align*}
b_1(s) &= n! (s+1)^n, \\
b_2(s) &= \prod_{1 \leq c < d \leq n} (d(s+1) + c), 
\end{align*}
such that $b_1 b_2 = \tilde b \alpha_n$.

\begin{prop}
The polynomial $b_1(s)$ divides the $b$-function of $f$.
\end{prop}
\begin{proof}

Consider the following local coordinates for $X,$
\begin{align*}
\bC^{n(n-1)}\times\ \bC^n\times \ \ \bC^n \quad &\longrightarrow X \\
\varphi: (\lbrace t^i_j \rbrace_{i \not = j},\lbrace a_i \rbrace_{i=1}^n,\lbrace v_i \rbrace_{i=1}^n) &\longmapsto (T A T^{-1}, T v),
\end{align*}
where 
\[
T = \left(
\begin{matrix}
1 & t^1_2 & t^1_3 & \ldots & t^1_n \\
t^2_1 & 1 & t^2_3 & \ldots & t^2_n \\
t^3_1 & t^3_2 & 1 & \ldots & t^3_n \\
\vdots & \vdots & \vdots & \ddots & \vdots \\
t^n_1 & t^n_2 & t^n_3 & \ldots & 1
\end{matrix}
\right)
\text{ and }
A = \left(
\begin{matrix}
a_1 & 0 & 0 & \ldots & 0 \\
0 & a_2 & 0 & \ldots & 0 \\
0 & 0 & a_3 & \ldots & 0 \\
\vdots & \vdots & \vdots & \ddots & \vdots \\
0 & 0 & 0 & \ldots & a_n
\end{matrix}
\right).
\]
Let $p \in X$ be defined by $t^i_j = 0, a_i = i $ and $v_i =0$. It is straightforward to compute that $\det(D\varphi)|_p \not = 0$, and thus $\varphi$ gives local coordinates at $p$. 

In these local coordinates  we have
\[
f = \det(T) \det(C(A,v)).
\]
Note $C(A,v)^i_j = v_i a_i^{j-1}$ is the Vandermonde matrix with rows multiplied by $v_i$.  So 
\[
\det(C(A,v)) = v_1 \cdots v_n \prod_{1 \leq i < j \leq n} (a_j - a_i).
\]
Let $\overline{1} = (1,\ldots,1)$. Since at $p$, $a_i = i$, the quantity 
\[
\det(C(A,\overline{1})) = \prod_{1 \leq i < j \leq n} (a_j - a_i) \not = 0.
\]
  Similarly, $T = I_n$ at $p$, so $\det(T) = 1$. Thus in a small open neighborhood $U$ of $p$, the function $\det(T) \det(C(A,\overline{1}))$ is invertible.  

We will show that the $b$-function for $f$ on $U$ is $b_1(s)$ which will finish the proof since the global $b$-function is the least common multiple of the local $b$-functions \cite[Proposition 2.1]{gran}. 

Let 
\[
D = \det(C(A,\overline{1}))^{-1} \det(T)^{-1} \partial_{v_1}\ldots \partial_{v_n}.
\]
Then
\begin{align*}
D f^{s+1} 
&= b_1(s) f^s.
\end{align*}
Since $v_i$ are not invertible near $p$, the Bernstein operator $D$ must include the factor $\partial_{v_1}\ldots \partial_{v_n}$.  Thus $b_1$ is minimal.
\end{proof}

\section{Radial Parts Reduction}
\label{sec-proof3}

For $b_2$, we relate our problem to the situation of \cite{Op5}.   To do this we require the radial parts map from \cite[Appendix]{GGS}.

We introduce more notation.  Denote by $R$ and $R^+$ the set of roots and positive roots respectively for type  $A_{n-1}$.  Let $W = S_n$ be the Weyl group corresponding to $\mathfrak{gl}_n$.  Consistent with the notation in \cite{GGS}, we define $X^{\reg}$ to consist of pairs $(M,v)$ where $M \in \gerg^{\mathrm{rs}}$.  Lastly define $\gerh^{\reg} \subset \gerh$ to be those points which avoid the root hyperplanes. 

\vspace{1pc}

The radial parts map is more clearly described in \cite{GGS}.  We give an overview.  To derive it, we start with the map
\begin{align*}
\rho: X^{\reg}/G &\longrightarrow \gerh^{\reg}/W \\
      (M,v)    &\longmapsto     \text{eigenvalues of } M,
\end{align*}
which induces the map
\[
\rho^*: \bC[\gerh^{\reg}]^W \longrightarrow \bC[X^{\reg}]^G.
\]
For $k \in \bC$, we define the radial parts map $\Rad_k: D_{X}^G \to D_{\gerh^{\reg}}^W$ as follows.  Let $D \in D_{X}^G$ and $g \in \bC[\gerh^{\reg}]^W$. Then
\[
\left.\Rad_k(D)(g) =  f^{-k} D (f^k \rho^*(g)) \right|_\gerh.
\]
Let $\delta = f|_{\gerh^{\reg}}$ be the Vandermonde determinant and $L_k$ be the \textit{Calogero-Moser} operator
\[ 
L_k = \Delta_{\gerh} - \sum_{\alpha\in R^+} k (k+1) \frac{(\alpha,\alpha)}{\alpha^2}.  %%CHANGE %I THINK GGS might be wrong.  I've removed a 1/2 here!
\]
Note that since $\delta$ is the product of the positive roots, $\alpha^{-1}$ as well as $\delta^{-1}$ are in $\bC[\gerh^\reg]$.
Let $\Delta_{\gerg}$ be the standard Laplacian on $\gerg$.  Since $X = \gerg \times V$, we can view $\Delta_{\gerg}$ as a differential operator on $X$.  By \cite[Appendix]{GGS},
\begin{equation}\label{eqn-cm-2}
\Rad_k(\Delta_{\gerg}) = \delta^{-k-1} L_k \delta^{k+1}.
\end{equation}
Define $P^+ = \sum_{\alpha \in R^+} \alpha^{-1} \partial_{X_\alpha}$.  Then we can simplify (\ref{eqn-cm-2}) using the equations $\delta^{-1} \Delta_{\gerh} \delta = \Delta_{\gerh} + 2 P^+ $ and $\delta^{-1} P^+ \delta =  P^+ + 2 \sum_{\alpha \in R^+} (\alpha,\alpha) \alpha^{-2}$ to get 
\begin{equation}\label{eqn-cm}
 \Rad_k(\Delta_{\gerg})  = \Delta_{\gerh} + 2 (k+1) P^+,  %%CHANGE % k -> k+1
\end{equation}
an operator which we will call $L_{\gerh}(k+1)$. 

  We say that $D  \in D_{\gerh^{\reg}}^W \otimes \bC [k]$ is a shift operator if   %%CHANGE  % k -> k+1
\begin{equation}\label{eqn-shift-op}
D L_{\gerh}(k) = L_{\gerh}(k+r) D .
\end{equation}

Denote the set of shift operators with shift $r$ as $\mathbb{S}_{\gerh}(r ,k)$ or just $\mathbb{S}_{\gerh}(r)$ if $k$ is clear from context.  Note that $\mathbb{S}_{\gerh}(0)$ is just the centralizer of $L_{\gerh}(k)$.

The following is a rational case analog to the trigonometric case given in \cite[Theorem 3.1]{Op5}.
\begin{prop} \label{prop-s-1d}
The set $\mathbb{S}_{\gerh}(r,k)$ is a free, rank-one $\mathbb{S}_{\gerh}(0,k)$-module.
\end{prop}

We defer the proof until Section \ref{Section-proof-prop}.  We denote the single generator of $\mathbb{S}_{\gerh}(r,k)$ over $\mathbb{S}_{\gerh}(0,k)$ by $g(r,k)$.  We denote the operator $p \mapsto f \cdot p$ by $\hat f$. 

\begin{prop}
The operator $\Rad_k(S \hat f)$ belongs to $\mathbb{S}_{\gerh}(-1,k+2)$.   % CHANGE % k+1 -> k+2
\end{prop}
\begin{proof}

As noted before, $S\hat f$ is a semi-invariant with character $\det \cdot \det^{-1} = 1$, so $S\hat f \in D_X^G$, and thus we can take its radial part $\Rad(S \hat f) \in D^W_{\gerh^{\reg}}$.

The set $\mathbb{S}_{\gerh}(-1,k+2)$ is a subset of $D_{\gerh^{\reg}}^W \otimes \bC [k].$  So we need to show $k \mapsto \Rad_k(S \hat f)$ is a polynomial map $\bC \to D_{\gerh^\reg}^W$.
Given $p \in \bC[X],$ we denote the corresponding differential operator as $\partial_{p}$ or $\partial( p )$. Let $p =  x_1 \cdots x_m $. Denote $(k+1)_r = (k+1)\ldots(k-r)$, the falling Pochhammer symbol. Then from the formula
\[
f^{-k} \partial_p f^{k+1} g = \sum_{q \cdot p_1 \cdots p_r = p} (k+1)_r f^{1-r} \left( \prod_{i=1}^r \partial_{p_i} f \right) \partial_q g,
\]
we can see $\Rad_k(S \hat f)$ is a polynomial in $k$, since $k$ appears only in the polynomial coefficients $(k+1)_r$.

Now we need to show that $\Rad_k(S \hat f)$ satisfies (\ref{eqn-shift-op}).  Since $S$ and $\Delta_{\gerg}$ are constant coefficient operators, they commute, which gives the equation
\[
S \hat f \widehat{f^{-1}} \Delta_{\gerg} \hat f = \Delta_{\gerg} S \hat f.
\]
Applying $\Rad_k$, which is a homomorphism, we get
\[
\Rad_k(S \hat f) \Rad_k( \widehat{f^{-1}} \Delta_{\gerg} \hat f )  = \Rad_k(\Delta_{\gerg}) \Rad_k(S \hat f).
\]
Thus we obtain
\[
\Rad_k(S \hat f) L_{\gerh}(k+2)  = L_{\gerh}(k+1) \Rad_k(S \hat f),   %CHANGE k -> k+1
\]
as required.
\end{proof}

The operator $\Rad_k(S \hat f)$ helps us compute $\shat b (k)$, since 
\begin{align*} \label{eqn-rad-Sf}
\Rad_k(S \hat f)(1) &= ( f^{-k} S  f^{k+1}  )|_{\gerh} = \shat b (k ).
\end{align*}

We now introduce the relation to the trigonometric case studied in \cite{Op5}.  Let $H$ be the complex torus with $\mathrm{Lie}(H) = \gerh$.  For $h \in H$, denote $h^\alpha = \exp(\alpha)(h)$. We then set
\[ 
H^\reg = \lbrace h \in H\ |\ h^\alpha \not = 0 \text{ for all } \alpha \in R^+  \rbrace.
\]
  Define the trigonometric Calogero-Moser operator with parameter $k \in \bC$ on $H^{\reg}:$
\[  
L_H(k) = \Delta_{\gerh} - \sum_{\alpha \in R^{+}} k \frac{1+h^\alpha}{1-h^\alpha} \partial(X_\alpha). 
\]
We also define
\[
\rho(k) = k \frac{1}{2} \sum_{\alpha \in R^+} \alpha.
\]  
Then we define the related operator $\tilde L_H(k) = L_H(k) + (\rho(k),\rho(k))$.  Let $\mathbb{S}_H(r,k)$ be the space of shift operators with respect to this trigonometric operator
\[
\mathbb{S}_H(r ,k) = \lbrace D \in D_{H}^W\ |\ D \tilde L_H(k) = \tilde L_H(k+r) D \rbrace.
\] 

The relationship to the rational case comes from the map $\epsilon: D_H \to D_\gerh$ which takes the lowest homogeneous part of the operator with respect to the grading defined by the \textit{degree} of the differential operator.  This map is more carefully defined in \cite[Section 3]{Op5}.  The lowest homogeneous degree or $\mathrm{lhd}$ of $\tilde L_H(k)$ is $-2,$ and so we have
\[
\epsilon(\tilde L_H(k)) = \Delta_{\gerh} + 2k \sum_{\alpha \in R^{+}}  \frac{1}{\alpha} \partial_{X_\alpha},
\]
which is $L_{\gerh}(k)$ from above.
Note that $\epsilon$ satisfies
\begin{align}
\epsilon(D_1 D_2) &= \epsilon(D_1) \epsilon(D_2), \quad \text{for all} \ D_1,D_2, \\
\epsilon(D_1 + D_2) &= \epsilon(D_1) + \epsilon(D_2) \quad \text{for all}\ D_1,D_2 \text{ with } \mathrm{lhd}(D_1) = \mathrm{lhd}(D_2).
\end{align}
Thus if $D \in \mathbb{S}_H(r,k),$ we get
\begin{align*}
\epsilon(D) L_{\gerh}(k) & = L_{\gerh}(k+r) \epsilon( D). 
\end{align*}
So $\epsilon(D) \in \mathbb{S}_{\gerh}(r,k)$.

From \cite[Theorem 3.1]{Op5} we know that $\mathbb{S}_H(r,k)$ is a rank-one $\mathbb{S}_H(0,k)$-module, similar to the case of $\mathbb{S}_{\gerh}$.  We denote the generator $G(r,k)$.  Moreover, we prove the following proposition in Section \ref{Section-proof-prop}. 

\begin{prop}\label{prop-ep-maps-Gtog}
The map $\epsilon$ sends $G(-1,k)$ to $g(-1,k)$.
\end{prop}

  Since $\Rad_k(S \hat f) \in \mathbb{S}_{\gerh}(-1,k+2),$ we know there exists some $D_0 \in \mathbb{S}_{\gerh}(0,k+2)$ such that $D_0 \cdot \epsilon(G(-1,k+2))  = \Rad_k(S\hat f)$. %%CHANGE k -> k+1
	By \cite[Theorem 3.3]{Op5}, we know $G(-1,k+2)$ has lowest homogeneous degree $0$.  Thus it can be written %%CHANGE k -> k+1
\[
G(-1,k+2) = \sum_{i\in\mathcal{I}} p_i \partial(q_i) %%CHANGE k -> k+1
\]
for some $p_i$ and $q_i$, homogeneous polynomials with $\deg(p_i) \geq \deg(q_i)$. Thus
\[
\epsilon(G(-1,k+2)) = \sum_i p_i \partial(q_i) %%CHANGE k -> k+1
\]
where the sum is over $\left\{ i \in \mathcal{I}\ | \deg(p_i) = \deg(q_i) \right\}$.
Given a differential operator $P$, we define the \textit{constant term} of $P$, denoted $\CT(P)$, to be scalar part of the operator, i.e., the summand with order zero and degree zero.   If we write $P= \sum_i p_i \partial(q_i)$, then $\CT(P) = p_k \partial(q_k)$ where $\deg(p_k) = \deg(q_k) = 0$.  We see that
\[
\epsilon(G(-1,k+2)) \cdot 1 = \CT(\epsilon(G(-1,k+2))) = \CT(G(-1,k+2)).  %%CHANGE k -> k+1
\]       

Let $r_0 = \CT(D_0) \in \bC[k]$. Then
\begin{align*}
 \CT(\Rad_k(Sf))  &=  \Rad_k(Sf).1 \\
                &= D_0 \epsilon( G(-1,k+2))) \cdot 1  \\ %%CHANGE k -> k+1
                &= D_0 \cdot \CT( G(-1,k+2))   \\ %%CHANGE k -> k+1
                &= r_0  \CT( G(-1,k+2)).  %%CHANGE k -> k+1
\end{align*}
We deduce that $\CT(G(-1,k+2))$ divides $\CT(\Rad_k(Sf)) = \shat b (k)$.  %%CHANGE k -> k+1
 
Let $\Gamma$ be the Gamma function.  For a reduced root system define
\[
\tilde c (\lambda,k) = \prod_{\alpha \in R^+} \frac{\Gamma(-(\lambda,\alpha^{\vee}))}{\Gamma(-(\lambda,\alpha^{\vee})+k)}.
\]
In the special case of type $A_n$ we have
\begin{equation}
\label{eqn-c-tilde}
\frac{1}{\tilde c(-\rho(k),k)} = \prod_{d=2}^n \frac{\Gamma(d k)}{\Gamma(k)}. 
\end{equation}
By \cite[Theorem 3.1, Corollary 3.4, Corollary 5.2]{Op5}, we know that
\begin{equation} \label{eqn-eta}
\CT(G(-1,k+2)) = \frac{\tilde c(-\rho(k+1), k+1)}{\tilde c (-\rho(k+2), k+2)}.  %%CHANGE k -> k+1
\end{equation}

So substituting (\ref{eqn-c-tilde}) into (\ref{eqn-eta}) and canceling, we arrive at
\begin{align*}
\CT(G(-1,k+2)) 
&= n! \prod_{d=2}^n \prod_{j=1}^{d-1} d(k+1)+j.  %%CHANGE k -> k+1
\end{align*}
In summary, we have proved the following.
\begin{prop}
 The polynomial $b_2(s) = \prod_{d=2}^n \prod_{j=1}^{d-1} d(s+1)+j$ divides $\shat b(s)$ in $\bC[s]$.  %%CHANGE s -> s+1
\end{prop}

Thus we have shown $b_1$ and $b_2$ divide $\shat b$.  Since $b_1$ and $b_2$ are coprime and $\deg(\tilde b) \leq \frac{n(n+1)}{2}.$ It follows that $\shat b =\tilde b \alpha_n = b_1 b_2$.

\section{Minimality of $\tilde b$}
\label{sec-proof4}

To complete the proof of Theorem \ref{thm-b-func}, we must show that $\tilde b$ has minimal degree.  In our case, this will follow directly from Proposition \ref{prop-minimal}.  Set $m = \frac{n(n+1)}{2} = \deg(f)$.

The proof that Sato's $b$-function in the prehomogeneous case gives the Bernstein polynomial (see,e.g., \cite{igusa}) is general enough to cover our situation.  The argument given in \cite{igusa} proves the following result.

\begin{prop}\label{prop-minimal}
Let $f \in \mathbf{C}[x_1,\ldots,x_n]$ be homogeneous of degree $d$.  If the operator $D \in \bC[s,\partial_{x_1},\ldots,\partial_{x_n}]$ satisfies (\ref{sato-bernstein}) for some polynomial $b_D(s)$ and is not divisible by any non-scalar factor in $\bC[s]$,  then $b_D$ is the Bernstein polynomial for $f$.  
\end{prop}   
\begin{proof}
The proof is identical to the proof of \cite[Theorem 6.3.1]{igusa}.  We give a sketch of it for the reader's convenience. 

For any polynomial $p(s) = \prod_{\mu} (s - \mu)$, we define
\[
\gamma_p(s) = \prod_{\mu} \Gamma(s - \mu).
\]

Let $b(s)$ be a polynomial satisfying (\ref{sato-bernstein}) with respect to $f \in \bC[X]$.  Define the zeta function corresponding to $f$ as
\[
Z_f(s) = \int_X |f(x)|^s e^{-2\pi|x|^2} dx.
\]
By \cite[Theorem 5.3.2]{igusa}, $Z_f(s) /\gamma_b(s)$ is holomorphic on $\bC$.  Each root of $b$ corresponds to infinitely many zeros of the function $Z_f(s) /\gamma_b(s)$.  Thus if $Z_f(s) /\gamma_b(s)$ is nowhere-vanishing, $b$ must be minimal.  We will now show this for $b_D$.

The integration by parts argument from \cite[Theorem 6.3.1]{igusa} applies using $D$ instead of $f(\partial)$ to show that
\[
Z_f(s) = (2 \pi)^{-d s} \gamma_{b_D}(s) \prod_{\mu}\frac{1}{\Gamma(-\mu)}
\]
for $\mathrm{Re}(s) > 0 $.  Thus $Z_f(s) / \gamma_{b_D}(s)$ is a nowhere-vanishing holomorphic function.  
\end{proof}

\begin{remark}
Let $D \in \bC[s,\partial_{x_1},\ldots,\partial_{x_n}]$ be an operator satisfying (\ref{sato-bernstein}) with respect to the Bernstein polynomial $b$. Let $\tilde D$ be the degree $-m$ part of $D$.  Then taking the degree $m s $ part of (\ref{sato-bernstein}) we get $\tilde D f^{s+1} = b(s) f^s$.   Since $G$ acts locally finitely on $D_X$, we can write $\tilde D = \sum_{i}\tilde D_i$ where $g \cdot \tilde D_i = \det(g)^i \tilde D_i$.  Then since $g f = \det(g) f$, this means $\tilde D_i f^{s+1} = 0$ unless $ i = -1$.  Thus we have $\tilde D_{-1} f^{s+1} = b(s) f^s$.  Then by Proposition $\ref{prop-f-structure}$ since $\deg(f) = \deg(S)$ we have 
\[
\tilde D_{-1} = c S
\] 
where $c \in \bC$. So we have, in fact, shown that any Bernstein operator independent of $x \in X$ realizing $b(s)$ is our differential operator $S$ up to a constant factor.
\end{remark}

%%%%%%%%%%%%%%%%%%%%%%%%%%%%%%%%%%%%%%%%%%%%%%%%%%%%%%%%%%%%%%%%%%%

\section{Appendix: Proof of Propositions \ref{prop-s-1d} and \ref{prop-ep-maps-Gtog} }
\label{Section-proof-prop}

In this section, we will prove Propositions \ref{prop-s-1d} and \ref{prop-ep-maps-Gtog}.

Given $\bj \in \mathbf{Z}^{R^+}$, an index on $R^{+}$, we define a partial order by $\bj \leq 0$ if $\bj_\alpha \leq 0$ for all $\alpha$.  We represent the basis of $\mathbf{Z}^{R^+}$ by $e_\alpha$ where $\alpha \in R^+$.  Also denote $\ba^\bj = \prod_{\alpha \in R^+} \alpha^{\bj_{\alpha}}.$  If $p \in \bC[k] \otimes \bC[\gerh],$ then as above, we denote the corresponding differential operator in $\bC[k] \otimes D_{\gerh}$ as $\partial_{p}$ or $\partial( p )$.

Let $h_1,\ldots,h_n$ be an orthonormal basis of $\gerh^*$. As a shorthand we write $\partial_i = \partial(h_i)$.  So $\Delta_{\gerh} = \sum_{i=1}^n \partial_i^2$.

The following lemma is a straightforward calculation.
\begin{lem}\label{lem-wa-calc-1}
Let $\sum \br = \sum_{i=1}^n \br_i$ and $p^{\left( \br \right)} = \prod_{i=1}^n \partial_{t_i}^{\br_i} p(t)$ and
\[
\AV(\br) = \frac{\left(\sum \br \right)!} {\prod_{i=1}^n \left( \br_i! \right) } (-1)^{\sum \br}.
\]
Then we have
\begin{equation}
\partial( p) (\alpha^{-1} e^{t\ld}) = \sum_{\br = 0}^{\binf} \AV(\br) \alpha^{-1-\sum \br} \prod_{i=1}^n \left( \partial_i \alpha \right)^{r_i} p^{\left( \br \right)}.
\end{equation}
\qed
\end{lem}

The key lemma we need for proving Proposition \ref{prop-s-1d} reads as follows.

\begin{lem}\label{lem-pn-determines}
The map
\begin{align*}
  p_\bN :       \mathbb{S}_{\gerh}(-1,k)  &\longrightarrow \bC[k] \otimes \bC[\gerh] \\
   \sum_{\bj \leq \bN} \ba^\bj \partial(p_\bj)&\longmapsto p_\bN
\end{align*}
is injective.
\end{lem}

\begin{proof}
First note that if $D \in \mathbb{S}_{\gerh}(r,k)$ and $S \in \mathbb{S}_{\gerh}(0,k)$, then commuting $S$ and $L_{\gerh}(k)$,
\begin{align*}
(D S) L_{\gerh}(k) &= L_{\gerh}(k+r) (D S).
\end{align*}
So $\mathbb{S}_{\gerh}(r,k)$ is an $\mathbb{S}_{\gerh}(0,k)$-module.  

We will solve for an arbitrary element $D$ of the set $\mathbb{S}_{\gerh}(r,k).$ Since $D \in \bC[k] \otimes D^W_{\gerh^{\reg}}$, we can write
\[
D = \sum_{\bj \leq \bN} \ba^\bj \partial(p_\bj) .
\] 
Let $ e^{t \ld} = e^{t_1 \lambda_1 + \ldots + t_n \lambda_n}$.  If $\alpha = \sum_i c_i \lambda_i,$ then denote $t_\alpha = \sum_i c_i t_i$.
Then applying (\ref{eqn-shift-op}) to $e^{t \ld}$, we solve for $D$ (i.e. solve for the $p_i$) in 
\[
D L_{\gerh}(k) e^{t \ld} = L_{\gerh}(k+r)D e^{t \ld}.
\]
Using Lemma \ref{lem-wa-calc-1}, we expand this into
\begin{align*}
 0 = D L_{\gerh}(k) e^{t \ld} & - L_{\gerh}(k+r) D e^{t \ld} = \\
e^{t \ld} \sum_{\bj \leq \bN} \ba^\bj  &\left[  2k \sum_{\alpha \in R^+} t_\alpha \sum_{\br = 0}^{\binf} \AV(\br) \alpha^{-1-\sum \br} \prod_i (\partial_i\ba)^{\br_i} p^{(\br)}(t)  \right. \\
    &  - p_\bj(t) 
    \left \lbrace
    	  \sum_{\alpha \in R^+} (\bj_\alpha^2 - \bj_\alpha) \alpha^{-2} \left(\sum_{l=1}^n (\partial_l \alpha)^2 \right) + 2\bj_\alpha \alpha^{-1} t_\alpha \right.  \\
	  &\ \ \ \ \ \ \ \ \ \ \ + 2 \bj_\alpha \sum_{\alpha \not = \beta} \bj_\beta \alpha^{-1} \beta^{-1} \left( \sum_{l=1}^n (\partial_l \alpha)(\partial_l \beta) \right) \\
	  &\ \ \ \ \ \ \ \ \ \ \left.\left.+ 2 (k+r) \left( \sum_{\alpha \in R^+} \alpha^{-2} \bj_\alpha 2 + \alpha^{-1} t_\alpha + 2 \sum_{\alpha \not = \beta} \beta^{-1} \alpha^{-1} \bj_\alpha (\partial_\beta \alpha)\right)
    \right\rbrace 
     \right].
\end{align*}
Hence extracting the coefficient of $\ba^{\bN - e_\beta} e^{t \ld} $, we get 
\[
0 = p_{\bN} t_\beta 2( k - \bN_\beta - (k+r))
\]
and so assuming that $p_\bN \not = 0$ we get that $\bN_\beta = -r$ for all $\beta$.  That is, $\bN = (-r,\ldots,-r)$.

Denote $\nabla \alpha = (\partial_1 \alpha, \ldots, \partial_m \alpha)$. In general, if we look at the coefficient of $\ba^\bM$ we get that the following expression is $0$:
\begin{align}\label{eqn-recursion}
&2k \sum_{\alpha \in R^+}  \sum_{\br = 0}^\binf p_{\bM+e_\alpha(1+\sum \br)}^{(\br)} t_\alpha  \AV(\br) \prod_{i=1}^n (\partial_i \alpha)^{\br_i}  \\ \nonumber
&- \sum_{\alpha \in R^+} p_{\bM+2 e_\alpha} (\bM_\alpha +2)(\bM_\alpha +1) \nabla \alpha \mathbf{\cdot} \nabla \alpha \\ \nonumber
&- \sum_{\alpha \in R^+} p_{\bM+ e_\alpha} 2 (\bM_\alpha+1) t_\alpha\\ \nonumber
&- \sum_{\alpha \not = \beta} p_{\bM+e_\alpha + e_\beta}  2 (\bM_\alpha+1)(\bM_\beta+1) \nabla \alpha \mathbf{\cdot} \nabla \beta\\ \nonumber
&- \sum_{\alpha \in R^+} p_{\bM+2 e_\alpha} 4(k+r) (\bM_\alpha+2)\\ \nonumber
&- \sum_{\alpha \in R^+} p_{\bM+ e_\alpha} 2 (k+r) t_\alpha\\ \nonumber
&- \sum_{\alpha \in R^+} p_{\bM+ e_\alpha + e_\beta} 4 (k+r) (\bM_\alpha+1) (\alpha, \beta). \nonumber
\end{align}
This equation involves only $p_{\bj}$ with $\bj \geq \bM$.  Define $\deg(\bj) = \sum_{\alpha \in R^+} \bj_\alpha$ and $|\bj| = |\deg(\bj)|$.  Then note that even though some of the $p_\bj$ are differentiated, none of the ones with $|\bM - \bj| = 1$ are.  We can use (\ref{eqn-recursion}) to solve for one of the $p_\bj$ in terms of $p_\bj$ of greater degree.  If we move through the $\ba^\bM$ with decreasing degree, we can thus show that all $p_\bj$ are ultimately determined by $p_\bN$ alone.  It is lengthy although not difficult to give a precise algorithm for how one should progress downward through the indices.

  Thus we know that $p_\bN$ determines $D$. 
  
\end{proof}

\begin{proof}[Proof of Proposition \ref{prop-ep-maps-Gtog}] 
We first note that since $L_{\gerh}(k)$ is homogeneous of degree $-2$, the generator of $\mathbb{S}_{\gerh}(r,k)$ must be homogeneous.  If $D \in \mathbb{S}_{\gerh}(r,k)$ is not homogeneous, then $D L_{\gerh}(k) = L_{\gerh}(k+r) D$ implies $\epsilon(D L_{\gerh}(k)) = \epsilon(L_{\gerh}(k+r) D)$ and since $L_{\gerh}(k)$ is already homogeneous, $\epsilon(D) L_{\gerh}(k) = L_{\gerh}(k+r) \epsilon(D)$.  So $\epsilon(D) \in \mathbb{S}_{\gerh}(r,k)$ but is not a multiple of $D$.  

Now since the degree of differential operators is additive, $g(-1,k)$ must have the minimal degree.  We know from \cite[Corollary 3.12]{OpIII} that $\deg(G(-1,k)) = \frac{n(n+1)}{2} ,$ and it has lowest homogeneous degree $0$ \cite[Theorem 4.4]{OpIII}.  Thus, since $\bN = (1,\ldots,1)$, we can only have $p_\bj \not = 0$ for $\bj$ where $\deg(p_\bj) = \sum \bj \geq 0.$  

Now we know $D \in D_\gerh^W$, so $w\cdot p_\bj(t) = p_{w.\bj}(t)$ where $W$ acts on the indices $\bj$ by permutation.  Thus $p_\bN$ is $W$-invariant.

Let $\alpha \in R^+$. By (\ref{eqn-recursion}) with $\bM = (1,\ldots,1) - 2 e_\alpha$, we see that 
\[
0=2k p_{\bN-e_\alpha} t_\alpha + \sum_{i=1}^n 2k (\partial_i p_{\bN}) t_\alpha \AV(e_i) (\partial_i \alpha) - 4(1-k) p_{\bN} - t_\alpha p_{\bN-e_\alpha} 2 (k-1) 
\]
since $\bM_\alpha = -1$.  Then we see that $t_\alpha$ divides $p_{\bN}$ unless
\[
0=2k p_{\bN-e_\alpha} t_\alpha + \sum_{i=1}^n 2k (\partial_i p_{\bN}) t_\alpha \AV(e_i) (\partial_i \alpha) - t_\alpha p_{\bN-e_\alpha} 2 (k-1). 
\]
This requires $p_{\bN}$ to be $0$, which makes $g(-1,k)=0$, contradicting the fact that it is a generator.

Thus for all $\alpha$, $t_\alpha$ divides $p_\bN$.  Then the positive roots give $\frac{n(n+1)}{2}$ independent divisors of $p_\bN$ and so $\deg(g(-1,k)) \geq n(n+1)/2$.  Thus the $p_\bN$ term of $g(-1,k)$ is a scalar multiple of the $p_\bN$ term of $\epsilon(G(-1,k))$.  So the operators are scalar multiples of each other and thus $\epsilon(G(-1,k))$ is a generator.  

\end{proof}

\begin{proof}[Proof of Proposition \ref{prop-s-1d}]
 By Lemma \ref{lem-pn-determines}, $D$ is determined by $p_\bN(D)$.  Given $D \in \mathbb{S}_{\gerh}(-1,k)$ and $S \in \mathbb{S}_{\gerh}(0,k)$, then the lead coefficients are multiplicative,
\[
p_\bN(SD) = p_\bN(S) p_\bN(D).
\]
Hence Lemma \ref{lem-pn-determines} reduces us to showing that $p_\bN(\mathbb{S}_{\gerh}(-1,k))$ is a rank-one $p_\bN(\mathbb{S}_{\gerh}(0,k))$-module.

By \cite[Theorem 1.7]{Heck} we know that 
\[C[t_1,\ldots,t_n]^W \subset p_\bN(\mathbb{S}_{\gerh}(0,k)).
\]
  By Proposition \ref{prop-ep-maps-Gtog}, we know that $\prod_{\alpha \in R^+} t_\alpha$ divides $p_\bN(D)$ for all $ D \in \mathbb{S}_{\gerh}(-1,k)$.  Now since $D$ is $W$-invariant and $\bN$ is fixed by $W$, we know that the lead term
  \[
  \left( \prod_{\alpha \in R^+} \alpha \right)\partial(p_\bN)
  \]
  is $W$-invariant.  Since $w$ acts on $\prod_{\alpha \in R^+}$ by $\sgn(w)$, this means $w$ acts on $p_\bN$ by $\sgn(w)$ and is thus equal to 
  \[
  q(t) \left( \prod_{\alpha \in R^+} \alpha \right)
  \] 
  where $q(t) \in C[t_1,\ldots,t_n]^W \subset p_\bN(\mathbb{S}_{\gerh}(0,k)).$  Since $p_\bN(\epsilon(G(-1,k))) = c \left(  \prod_{\alpha \in R^+} \alpha \right)$, we know $\left(  \prod_{\alpha \in R^+} \alpha \right) \in p_\bN(\mathbb{S}_{\gerh}(-1,k))$.  Thus 
  \[
  p_\bN(\mathbb{S}_{\gerh}(-1,k)) = C[t_1,\ldots,t_n]^W \left(  \prod_{\alpha \in R^+} \alpha \right)     
  \]
  is clearly a rank-one $p_\bN(\mathbb{S}_{\gerh}(0,k))$-module.
\end{proof}

\hfill 
\bibliographystyle{amsplain}
\bibliography{refs}

\def\cprime{$'$} \def\cprime{$'$}
\providecommand{\bysame}{\leavevmode\hbox to3em{\hrulefill}\thinspace}
\providecommand{\MR}{\relax\ifhmode\unskip\space\fi MR }
\providecommand{\MRhref}[2]{%
  \href{http://www.ams.org/mathscinet-getitem?mr=#1}{#2}
}
\providecommand{\href}[2]{#2}
\begin{thebibliography}{10}

\bibitem{Bern}
I.~N. Bern{\v{s}}te{\u\i}n, \emph{Modules over a ring of differential
  operators. {A}n investigation of the fundamental solutions of equations with
  constant coefficients}, Funkcional. Anal. i Prilo\v zen. \textbf{5} (1971),
  no.~2, 1--16.

\bibitem{mir3}
M.~Finkelberg and V.~Ginzburg, \emph{Cherednik algebras for algebraic curves},
  Representation theory of algebraic groups and quantum groups, Progr. Math.,
  vol. 284, Birkh\"auser/Springer, New York, 2010, pp.~121--153.

\bibitem{mir1}
\bysame, \emph{On mirabolic {$\mathcal D$}-modules}, Int. Math. Res. Not. IMRN
  (2010), no.~15, 2947--2986.

\bibitem{mir2}
W.~L. Gan and V.~Ginzburg, \emph{Almost-commuting variety, {$\mathcal
  D$}-modules, and {C}herednik algebras}, IMRP Int. Math. Res. Pap. (2006),
  1--54.

\bibitem{GGS}
V.~Ginzburg, I.~Gordon, and J.~T. Stafford, \emph{Differential operators and
  {C}herednik algebras}, Selecta Math. (N.S.) \textbf{14} (2009), no.~3-4,
  629--666.

\bibitem{gran}
M.~Granger, \emph{Bernstein-{S}ato polynomials and functional equations},
  Algebraic approach to differential equations, World Sci. Publ., Hackensack,
  NJ, 2010, pp.~225--291.

\bibitem{Heck}
G.~J. Heckman, \emph{A remark on the {D}unkl differential-difference
  operators}, Harmonic analysis on reductive groups ({B}runswick, {ME}, 1989),
  Progr. Math., vol. 101, Birkh\"auser Boston, Boston, MA, 1991, pp.~181--191.

\bibitem{igusa}
J.~Igusa, \emph{An introduction to the theory of local zeta functions}, AMS/IP
  Studies in Advanced Mathematics, vol.~14, American Mathematical Society,
  Providence, RI, 2000.

\bibitem{kashiwara1976b}
M.~Kashiwara, \emph{{$B$}-functions and holonomic systems. {R}ationality of
  roots of {$B$}-functions}, Invent. Math. \textbf{38} (1976/77), no.~1,
  33--53.

\bibitem{OpIII}
E.~M. Opdam, \emph{Root systems and hypergeometric functions. {III}},
  Compositio Math. \textbf{67} (1988), no.~1, 21--49.

\bibitem{Op5}
\bysame, \emph{Some applications of hypergeometric shift operators}, Invent.
  Math. \textbf{98} (1989), no.~1, 1--18.

\bibitem{sato1}
M.~Sato, \emph{Theory of prehomogeneous vector spaces (algebraic part)---the
  {E}nglish translation of {S}ato's lecture from {S}hintani's note}, Nagoya
  Math. J. \textbf{120} (1990), 1--34, Notes by Takuro Shintani, Translated
  from the Japanese by Masakazu Muro.

\bibitem{sato2}
M.~Sato and T.~Shintani, \emph{On zeta functions associated with prehomogeneous
  vector spaces}, Ann. of Math. (2) \textbf{100} (1974), 131--170.

\bibitem{YS}
T.~Yano and J.~Sekiguchi, \emph{The microlocal structure of weighted
  homogeneous polynomials associated with {C}oxeter systems. {I}}, Tokyo J.
  Math. \textbf{2} (1979), no.~2, 193--219.

\end{thebibliography}

\end{document}